\documentclass{article}

\usepackage{amsmath}
\usepackage{amsthm}
\usepackage{amsfonts}
\usepackage{amssymb}
\usepackage{graphicx}
\usepackage{enumerate}
\usepackage{enumitem}

\usepackage[affil-it]{authblk}

\DeclareMathOperator*{\Int}{int}
\DeclareMathOperator*{\relint}{relint}
\DeclareMathOperator*{\conv}{conv}

\newtheorem{thm}{Theorem}[section]
\newtheorem{cor}[thm]{Corollary}
\newtheorem{con}[thm]{Conjecture}
\newtheorem{lem}[thm]{Lemma}

\theoremstyle{definition}
\newtheorem{defn}[thm]{Definition}
\newtheorem{prob}{Problem}
\theoremstyle{remark}
\newtheorem{rem}[thm]{Remark}

\begin{document}

\title{On the multiple illumination numbers of convex bodies}

\author{Kirati Sriamorn\footnote{kirati.s@chula.ac.th}}

\affil{Department of Mathematics and Computer Science, Faculty of Science, Chulalongkorn University}

\maketitle

\begin{abstract}
  In this paper, we introduce an $m$-fold illumination number $I^m(K)$ of a convex body $K$ in Euclidean space $\mathbb{E}^d$, which is the smallest number of directions required to $m$-fold illuminate $K$, i.e., each point on the boundary of $K$ is illuminated by at least $m$ directions. We get a lower bound of $I^m(K)$ for any $d$-dimensional convex body $K$, and get an upper bound of $I^m(\mathbb{B}^d)$, where $\mathbb{B}^d$ is a $d$-dimensional unit ball. We also prove that $I^m(K)=2m+1$, for a $2$-dimensional smooth convex body $K$. Furthermore, we obtain some results related to the $m$-fold illumination numbers of convex polygons and cap bodies of $\mathbb{B}^d$ in small dimensions. In particular, we show that $I^m(P)=\left\lceil mn/{\left\lfloor\frac{n-1}{2}\right\rfloor}\right\rceil$, for a regular convex $n$-sided polygon $P$.
\end{abstract}

\bigskip

\textbf{Keywords} Multiple illumination $\cdot$ Illumination number$\cdot$ Hadwiger conjecture $\cdot$  Smooth boundary $\cdot$ Cap body $\cdot$ Spiky body $\cdot$ Regular convex polygon

\medskip

\textbf{Mathematics Subject Classification} 52A20 $\cdot$  52A55 $\cdot$ 52C17

\section{Introduction}
Consider a convex body $K$ in $d$-dimensional Euclidean space $\mathbb{E}^d$, a point $p$ on its boundary $\partial K$, and a point $u$ on the origin-centered $(d-1)$-unit sphere $\mathbb{S}^{d-1}$. We say that $u$ illuminates $p$ with respect to $K$ if and only if there is a point $q$ in the interior $\Int(K)$ of $K$, such that $q=p+\lambda u$ for some positive number $\lambda$. For a positive integer $m$, a convex body $K$ is $m$-fold illuminated by a multiset $U$ of directions in $\mathbb{S}^{d-1}$, if each point in $\partial K$ is illuminated by at least $m$ directions from $U$, with respect to $K$. The $m$-fold illumination number $I^m(K)$ of $K$ is the smallest number of directions required to $m$-fold illuminate $K$.

For $m=1$, it is clear that $I^1(K)$ is the ordinary illumination number $I(K)$ which was first given by Boltyanskii\cite{Boltyanskii}. The following conjecture is well-known.

\begin{con}
	(Illumination conjecture) The illumination number $I(K)$ of any $d$-dimensional convex body $K$ does not exceed $2^d$. Moreover, $I(K))=2^d$ if and only if $K$ is an affine image of a $d$-cube.
\end{con}

The illumination problem has been completely solved only in $\mathbb{E}^2$ (\cite{Boltyanskii},\cite{Gohberg},\cite{Hadwiger} and \cite{Levi}). For three- dimensional space $\mathbb{E}^3$, Lassak\cite{Lassak} proved that the conjecture is true for any centrally symmetric convex body $K$. For the non-symmetric case, the best result is due to Papadoperakis's work\cite{Papadoperakis}:  $I(K)\leq 16$, for any three-dimensional convex body $K$. In higher dimensions,  Prymak\cite{Prymak} recently showed that for any convex body $K$ in $ \mathbb{E}^4, \mathbb{E}^5$, and $\mathbb{E}^6$, we have $I(K)\leq 96$, $I(K)\leq 1091$, and $I(K)\leq 15373$, respectively. For a long time ago, the best general estimate for $d$-dimensional space was due to Rogers\cite{Rogers}:
$$I(K)\leq\binom{2d}{d}(\ln d+\ln\ln d+5)=O(4^d\sqrt{d}\ln d).$$ 
In 2018, Huang\cite{Huang} and his collaborations improved this bound to $$I(K)\leq c_14^de^{-c_2\sqrt{d}}.$$
for some universal constants $c_1$ and $c_2$. 

For a positive integer $m$ and a convex body $K$. One can see that if $K$ is ($1$-fold) illuminated by $\{u_1,.\ldots,u_n\}\subset \mathbb{S}^{d-1}$, then $K$ is $m$-fold illuminated by  the multiset $\{m\cdot u_1,\ldots, m\cdot u_n\}$, where $m$ is the multiplicity of $u_i$ for $i=1,\ldots,n$. It follows that $$I^m(K)\leq m\cdot I(K).$$ 
In general, we have
$$I^{m+n}(K)\leq I^m(K)+I^n(K).$$
It is easy to see that $I^m(K)=2^dm$, when $K$ is an affine image of a $d$-cube. Analogous to the illumination conjecture, we have the following conjecture.
\begin{con}
	For a positive integer $m$ and a $d$-dimensional convex body $K$,
	$$I^m(K)\leq 2^dm$$
	and $I^m(K)=2^dm$ if and only if $K$ is an affine image of a $d$-cube.
\end{con}

Denote by $\partial K$ and $\Int(K)$ the boundary and the interior of $K$, respectively. Hadwiger\cite{Hadwiger} has offered a point source interpretation of the illumination problem. Let $q$ be a point on the boundary $\partial K$, and let $p\in\mathbb{E}^d\setminus K$. We say that the point $p$ illuminates $q$ with respect to $K$, if the ray $pq$ intersects with the interior $\Int(K)$, and the line segment $\overline{pq}$ does not intersect $\Int(K)$. Denote by $I_p(K)$ the smallest number of point sources that completely illuminate $\partial K$. The illumination problems are also related to the covering problems proposed by Hadwiger, Markus and Gohberg. Hadwiger\cite{Hadwiger_cover} studied the smallest number of translates of $\Int(K)$ required to cover $K$, denote by $C(K)$. Later, Markus and Gohberg\cite{Gohberg} studied the smallest number of smaller homothetic copies of $K$ required to cover $K$, denote by $C_h(K)$. It is well-known that\cite{Boltyanskii2}
$$I(K)=I_p(K)=C(K)=C_h(K),$$
for any convex body $K$. Similar to $I^m(K)$, one can give the definitions of $I_p^m(K), C^m(K)$, and $C_h^m(K)$, respectively. We call $C^m(K)$ the $m$-fold covering number of $K$. Then we also have
$$I^m(K)=I_p^m(K)=C^m(K)=C_h^m(K).$$

Let $\mathbb{B}^d$ be the $d$-dimensional unit ball with center at the origin $o$. It is known that $I(\mathbb{B}^d)=d+1$. In fact, if the boundary of $K$ is smooth, i.e., there is a unique support hyperplane of $K$ at each boundary point of $K$, then $I(K)=d+1$\cite{Hadwiger_smooth}.

In \cite{Minkowski}, \cite{Naszodi}, \cite{Ivanov}, and \cite{Bezdek}, the authors studied the convex bodies named cap bodies and spiky balls, which are the union of convex hulls of $\mathbb{B}^d$ and a point $v\in V$, where $V$ is a finite subset of $\mathbb{E}^d\setminus\mathbb{B}^d$. We can generalize these concepts as follows.

\begin{defn}
	Let $V$ be a finite subset of $\mathbb{E}^d\setminus K$. We call the set $$\widehat{K}(V)=\bigcup_{v\in V}\conv(K\cup\{v\})$$
	a spiky body of $K$ with respect to $V$, where $\conv(\cdot)$ refers to the convex hull of the corresponding set.
\end{defn}
Note that a spiky body $\widehat{K}(V)$ may not be convex. For convenience, if $V=\{v\}$, then we define $\widehat{K}(v)=\widehat{K}(V)$.

\begin{defn}\label{def_cap_body}
	If $V$ is a countable set, and for any pair of distinct points $v_1,v_2\in V$, the line segment $\overline{v_1v_2}$ intersects $K$, then $\widehat{K}(V)$ is called a cap body of $K$ with respect to $V$.
\end{defn}

In this paper, we will prove the following results (see Section \ref{subsection_lower_bound}, \ref{subsection_smooth} and \ref{subsection_upper_bound_ball}, respectively).

\begin{thm}\label{main_lower_bound}
	Let $m$ and $d$ be positive integers, where $d\geq 2$. For any $d$-dimensional convex body $K$, we have
	$$ I^m(K)\geq 2m+(d-1).$$
\end{thm}

\begin{thm}\label{main_smooth_two_dim}
	Let $m$ be a positive integer. Let $K$ be a $2$-dimensional convex body with smooth boundary. Then
	$$ I^m(K)= 2m+1.$$
\end{thm}

\begin{thm}\label{main_ball_upper_bound}
	For any positive integers $m$ and $d$, where $d\geq 3$, we have
	$$I^m(\mathbb{B}^d)\leq (d-1)m+1+\left\lceil\frac{m}{2}\right\rceil.$$
\end{thm}

Furthermore, we also obtain some results related to the illumination numbers of convex polygons (Section \ref{subsection_polygon}) and cap bodies of $\mathbb{B}^d$ in small dimensions (Section \ref{subsection_cap_body_of_ball}). In particular, we get the $m$-fold illumination numbers of regular convex polygons as follows.

\begin{thm}\label{main_regular_polygon}
	Let $K$ be a regular convex $n$-sided polygon, where $n\geq 3$. Then
	$$I^m(K)=\left\lceil\frac{mn}{\left\lfloor\frac{n-1}{2}\right\rfloor}\right\rceil.$$
	In addition, if $n$ satisfies one of the following conditions:
	\begin{enumerate}[label=(\roman*)]
		\item $n$ is odd and $n\geq 2m+1$,
		\item $n$ is even and $n\geq 4m+2$,
	\end{enumerate}
	then $I^m(K)=2m+1$.
\end{thm}

\section{Useful lemmas}
 In this section, we will study some properties of cap bodies $\widehat{K}(V)$, which will be very useful for proving our results.
 
	For any $p\in \widehat{K}(V)$, there exists a point $v\in V$ such that
	$$p\in\conv(K\cup\{v\}).$$
	Since $K$ is convex, we have that there is a point $q\in K$ such that
	$$p=\alpha\cdot q+\beta\cdot v,$$
	for some $\alpha,\beta\geq 0$, and $\alpha+\beta=1$.

\begin{lem}\label{spiky_convex_condition}
   Suppose that $V$ is a finite subset of $\mathbb{E}^d\setminus K$ such that for any pair of distinct points $v_1,v_2\in V$, the line segment $\overline{v_1v_2}$ intersects $K$. Then	$\widehat{K}(V)$ is convex and
	$$\widehat{K}(V)=\conv(K\cup V).$$
\end{lem}
\begin{proof}
	Suppose that $p_1,p_2\in\widehat{K}(V)$ , then
	$$p_1=\alpha_1\cdot q_1+\beta_1\cdot v_1$$
	and
	$$p_2=\alpha_2\cdot q_2+\beta_2\cdot v_2,$$
	for some $q_1,q_2\in K$, $v_1,v_2\in V$, and $\alpha_1,\alpha_2, \beta_1, \beta_2\geq0$, where $\alpha_1+\beta_1=1$ and $\alpha_2+\beta_2=1$. 

	For any $\alpha,\beta\geq0$, where $\alpha+\beta=1$. We have
	$$\alpha\cdot p_1+\beta\cdot p_2=\alpha\alpha_1\cdot q_1+\beta\alpha_2\cdot q_2+\alpha\beta_1\cdot v_1+\beta\beta_2\cdot v_2.$$
	If $v_1=v_2$, then it is clear that $\alpha\cdot p_1+\beta\cdot p_2\in \conv(K\cup\{v_1\})\subset\widehat{K}(V)$.
	Now suppose that $v_1\neq v_2$. Since the line segment $\overline{v_1v_2}$ intersects $K$, there is a point $q\in K$ such that $q=\alpha'\cdot v_1+\beta'\cdot v_2$, for some $\alpha',\beta'>0$ and $\alpha'+\beta'=1$.
	If $\alpha\beta_1\beta'=\beta\beta_2\alpha'$, then
	$$\alpha\cdot p_1+\beta\cdot p_2=\alpha\alpha_1\cdot q_1+\beta\alpha_2\cdot q_2+\alpha\beta_1(\alpha')^{-1}\cdot q\in K\subset \widehat{K}(V).$$
	If $\alpha\beta_1\beta'>\beta\beta_2\alpha'$, then
	\begin{align*}
		\alpha\cdot p_1+\beta\cdot p_2 &=\alpha\alpha_1\cdot q_1+\beta\alpha_2\cdot q_2+\beta\beta_2(\beta')^{-1}\cdot q+(\alpha\beta_1-\beta\beta_2(\beta')^{-1}\alpha')\cdot v_1\\
		&\in \conv(K\cup\{v_1\})\subset \widehat{K}(V).
	\end{align*}
	If $\alpha\beta_1\beta'<\beta\beta_2\alpha'$, then
	\begin{align*}
		\alpha\cdot p_1+\beta\cdot p_2 &=\alpha\alpha_1\cdot q_1+\beta\alpha_2\cdot q_2+\alpha\beta_1(\alpha')^{-1}\cdot q+(\beta\beta_2-\alpha\beta_1(\alpha')^{-1}\beta')\cdot v_2\\
		&\in \conv(K\cup\{v_2\})\subset \widehat{K}(V).
	\end{align*}
	It follows that $\widehat{K}(V)$ is convex, and hence
	$$\conv(K\cup V)\subset \widehat{K}(V).$$
	On the other hand, it is clear that
	$$\widehat{K}(V)=\bigcup_{v\in V}\conv(K\cup\{v\})\subset\conv(K\cup V).$$ 
	Therefore,
	$$\widehat{K}(V)=\conv(K\cup V).$$ 
\end{proof}
\begin{rem}
	Lemma \ref{spiky_convex_condition} is also true, provided $V$ is countable.  From Definition \ref{def_cap_body}, we know that a cap body is always convex.
\end{rem}

\begin{defn}
	Let $K$ be a $d$-dimensional convex body, and let $v\in\mathbb{E}^d\setminus K$. We call the set $\widehat{K}(v)\setminus K$ a spike of $K$ with respect to the point $v$, and denote by $S(K,v)$.
\end{defn}

\begin{lem}\label{spiky_body_subset}
	Let $K$ be a $d$-dimensional convex body, and let $v,v'\in\mathbb{E}^d\setminus K$. If $v'\in S(K,v)$, then
	$$\widehat{K}(v')\subset\widehat{K}(v).$$
\end{lem}
\begin{proof}
Since $\widehat{K}(v)$ is convex, and $v'\in S(K,v)\subset \widehat{K}(v)$ and $K\subset \widehat{K}(v)$, we have that $\widehat{K}(v')=\conv(K\cup\{v'\})\subset\widehat{K}(v)$.
\end{proof}

\begin{defn}
	Let $K$ be a $d$-dimensional convex body, and let $v\in\mathbb{E}^d\setminus K$. We call the set
	$$\left\{p\in\partial K\mid p=\alpha\cdot q+(1-\alpha)\cdot v,~\mbox{for some}~q\in\Int(K), \alpha\in(0,1) \right\}$$ 
	an open cap of $K$ with respect to the point $v$, and denote by $C(K,v)$.
\end{defn}

In fact, the open cap $C(K,v)$ is the collection of all boundary points of $K$ that are illuminated by the point $v$ with respect to $K$, i.e.,
$$C(K,v)=\{p\in\partial K\mid p~\mbox{is illuminated by the point}~v~\mbox{with respect to}~K\}.$$

\begin{lem}\label{int_cap_body}
	If $p\in\Int(\widehat{K}(v))$ and $p\not\in \Int(K)$, then there exist $q\in\Int(K)$ and $\alpha\in(0,1)$ such that
	$$p=\alpha\cdot q+(1-\alpha)\cdot v.$$
\end{lem}
\begin{proof}
	Since $p\in\Int(\widehat{K}(v))\subset \widehat{K}(v)$, we have that the ray $vp$ intersects $K$. Suppose, for the sake of contradiction, that the ray $vp$ does not intersect $\Int(K)$. Then the line $vp$ does not intersect $\Int(K)$. By the hyperplane separation theorem, there is a hyperplane $H$ such that $H$ contains the line $vp$, and $K$ lies in a closed half space $H^+$ that contains the hyperplane $H$ as its boundary. Since $v\in H\subset H^+$ and $K\subset H^+$, we have that $\widehat{K}(v)\subset H^+$. Hence, $p\in \Int(\widehat{K}(v))\subset \Int(H^+)$. This is a contradiction, since $p\in H$.
	Therefore, the ray $vp$ intersects $\Int(K)$. Suppose that $q\in\Int(K)$ and $q$ lies on the ray $vp$. Then there exists $\alpha>0$ such that
	$$p=\alpha\cdot q+(1-\alpha)\cdot v.$$
	It is clear that $\alpha\neq 1$, since $p\not\in \Int(K)$. From $p\in \widehat{K}(v)$, there exist $q'\in K$ and $\lambda\in[0,1]$ such that
	$$p=\lambda\cdot q'+(1-\lambda)\cdot v.$$
	If $\alpha> 1$, then we obtain
	$$ p=\frac{(1-\lambda)\alpha}{\alpha-\lambda}\cdot q+\frac{(\alpha-1)\lambda}{\alpha-\lambda}\cdot q'\in\Int(K),$$
	which is a contradiction. Hence, $\alpha\in(0,1)$.
\end{proof}

By Lemma \ref{int_cap_body}, one can deduce that
\begin{cor}\label{cap_equal_to_int_spkie}
	$C(K,v)=\partial K\cap\Int(\widehat{K}(v))$.
\end{cor}

\begin{lem}\label{illuminate_cap}
		Let $K$ be a $d$-dimensional convex body, and let $v\in\mathbb{E}^d\setminus K$. For any given $u\in\mathbb{S}^{d-1}$. If the direction $u$ illuminates the point $v$ with respect to the cap body $\widehat{K}(v)$, then 
		\begin{enumerate}[label=(\roman*)]
			\item $u$ illuminates each point in $S(K,v)\cap\partial(\widehat{K}(v))$ with respect to $\widehat{K}(v)$,
			\item $u$ illuminates each point in $C(K,v)$ with respect to $K$.
		\end{enumerate}
\end{lem}
\begin{proof}
	Since the direction $u$ illuminates the point $v$ with respect to $\widehat{K}(v)$, there exist $p\in \Int(\widehat{K}(v))$ and $\lambda>0$ such that $p=v+\lambda\cdot u$. From Lemma \ref{int_cap_body}, there exist $q\in\Int(K)$ and $\alpha\in(0,1)$ such that
	$$p=\alpha\cdot q+(1-\alpha)\cdot v.$$ 
	It follows that 
	$$u=\alpha\lambda^{-1}\cdot(q-v).$$
	Suppose that $s\in S(K,v)\cap\partial(\widehat{K}(v))$. Then there exist $t\in K$ and $\beta\in[0,1]$ such that
	$$s=(1-\beta)\cdot t+\beta\cdot v.$$
	If $\beta=0$ then $s=t\in K$ which is impossible, since $s\in S(K,v)=\widehat{K}(v)\setminus K$. Let $\varepsilon=\beta \lambda\alpha^{-1}$. Then
	$$s+\varepsilon\cdot u=(1-\beta)\cdot t+\beta\cdot q\in \Int(K)\subset\Int(\widehat{K}(v)).$$
	Hence, $u$ illuminates $s$ with respect to $\widehat{K}(v)$.
	
	Now consider any point $z\in C(K,v)$. Then $z\in\partial K$ and
	$$z=(1-\gamma) \cdot y+\gamma\cdot v,$$
	for some $y\in\Int(K)$ and $\gamma\in(0,1)$. Let $\delta=\gamma\lambda\alpha^{-1}$. We have
	$$z+\delta\cdot u=(1-\gamma)\cdot y+\gamma\cdot q\in\Int(K).$$
	Therefore, $u$ illuminates $z$ with respect to $K$.
\end{proof}

By adapting the arguments used in the first part of Lemma \ref{illuminate_cap}, one can obtain the following lemma.
\begin{lem}\label{illuminate_inside_spike}
	Let $K$ be a $d$-dimensional convex body, and let $v\in\mathbb{E}^d\setminus K$. For any given $u\in\mathbb{S}^{d-1}$. If the direction $u$ illuminates the point $v$ with respect to the spiky body $\widehat{K}(v)$, then for any $s\in S(K,v)$, the direction $u$ illuminates the point $s$ with respect to $\widehat{K}(s)$.
\end{lem}
\begin{proof}
		Since $u$ illuminates the point $v$ with respect to $\widehat{K}(v)$, there exist $p\in \Int(\widehat{K}(v))$ and $\lambda>0$ such that $p=v+\lambda\cdot u$. From Lemma \ref{int_cap_body}, there exist $q\in\Int(K)$ and $\alpha\in(0,1)$ such that
	$$p=\alpha\cdot q+(1-\alpha)\cdot v.$$ 
	It follows that 
	$$u=\alpha\lambda^{-1}\cdot(q-v).$$
	For any $s\in S(K,v)$, there exist $t\in K$ and $\beta\in[0,1]$ such that
	$$s=(1-\beta)\cdot t+\beta\cdot v.$$
	If $\beta=0$ then $s=t\in K$ which is impossible, since $s\in S(K,v)=\widehat{K}(v)\setminus K$. Let $\varepsilon=\beta \lambda \alpha^{-1}$. Then
	$$s+\varepsilon\cdot u=(1-\beta)\cdot t+\beta\cdot q\in \Int(K)\subset\Int(\widehat{K}(s)).$$
	Hence, $u$ illuminates $s$ with respect to $\widehat{K}(s)$.
\end{proof}

\begin{defn}
	We call the set
	$$C(K,v)\cup\{p\in\partial K\mid ~\mbox{the ray}~vp~\mbox{does not intersect}~\Int(K)\}$$
	a closed cap of $K$ respect to $v$, and denote by $\overline{C}(K,v)$.
\end{defn}

\begin{lem}\label{cap_subset}
	Let $K$ be a $d$-dimensional convex body, and let $v,v'\in\mathbb{E}^d\setminus K$. If $v'\in S(K,v)$, then
	$$C(K,v')\subset C(K,v)~\mbox{and}~\overline{C}(K,v')\subset\overline{C}(K,v).$$
\end{lem}
\begin{proof}
	By Lemma \ref{spiky_body_subset} and Corollary \ref{cap_equal_to_int_spkie}, we have
	$$C(K,v')=\partial K\cap \Int(\widehat{K}(v'))\subset\partial K\cap \Int(\widehat{K}(v))=C(K,v).$$
	Now we suppose that $p\in\partial K$ and the ray $v'p$ does not intersect $\Int(K)$. Then the line $v'p$ does not intersect $\Int(K)$. Let $H$ be a support hyperplane of $K$ at the point $p$ such that $H$ passes through $v'$. If the ray $vp$ does not intersect $\Int(K)$, then $p\in\overline{C}(K,v)$. If the ray $vp$ intersects $\Int(K)$, then $v$ does not belong to $H$. Since $v'\in S(K,v)$, we have that $v$ and $K$ lie on the opposite sides of the hyperplane $H$. It follows that the line segment $\overline{vp}$ does not intersect $\Int(K)$. Therefore, $p$ is illuminated by $v$ with respect to $K$, i.e., $p\in C(K,v)\subset\overline{C}(K,v)$.
	
\end{proof}

\begin{lem}\label{illuminate_closed_cap}
	Suppose that the boundary of $K$ is smooth. Let $v\in\mathbb{E}^d\setminus K$.  For any given $u\in\mathbb{S}^{d-1}$. If the direction $u$ illuminates the point $v$ with respect to the cap body $\widehat{K}(v)$, then $u$ illuminates each point in $\overline{C}(K,v)$ with respect to $K$.
\end{lem}

\begin{proof}
	From the proof of Lemma \ref{illuminate_cap}, we know that there exist $q\in\Int(K)$ and $\lambda>0$ such that
	$$u=\lambda\cdot(q-v).$$
	By Lemma \ref{illuminate_cap}, it suffices to show that the direction $u$ illuminates each point in
	$$\{p\in\partial K\mid ~\mbox{the ray}~vp~\mbox{does not intersect}~\Int(K)\},$$
	with respect to $K$. Suppose that $p\in\partial K$, and the ray $vp$ does not intersect $\Int(K)$. Then the line $vp$ does not intersect $\Int(K)$.  Let $H$ be the support hyperplane of $K$ at the point $p$. Since $K$ is smooth, we get $v\in H$. Suppose that $H^+$ is the closed half space that contains $K$, and $\partial H^+=H$. Since $v,p\in H$ and
	$v+\lambda^{-1}\cdot u=q\in\Int(K)\subset \Int(H^{+})$, we have
	$$p+\lambda^{-1}\cdot u\in \Int(H^{+}).$$
	Let $L$ be the line passing through $p$ and parallel to $u$. Then $L$ is not parallel to $H$.   Since $K$ has a unique support hyperplane at $p$, one obtains that $L$ intersects $\Int(K)$. This implies that the direction $u$ illuminates $p$ with respect to $K$. 
\end{proof}

\begin{rem}\label{replace_smooth_condition}
	In fact, in Lemma \ref{illuminate_closed_cap}, even when the boundary $K$ is not smooth,  we still can prove that the direction $u$ illuminates a  point $q$ in $\{p\in\partial K\mid ~\mbox{the ray}~vp~\mbox{does not intersect}~\Int(K)\}$ with respect to $K$, if $K$ has a unique support hyperplane at the point $q$.
\end{rem}

\begin{lem}\label{possible_positions_of_point_on_boundary_of_K}
 Let $K$ be a $d$-dimensional convex body. Let $V$ be a finite subset of $\mathbb{E}^d\setminus K$ that satisfies the condition in Lemma \ref{spiky_convex_condition}. Suppose that $p\in\partial K$. Then one of the following statements is true.
 \begin{enumerate}[label=(\roman*)]
 	\item $p\not\in \partial S(K,v)$ for all $v\in V$.
 	\item $p\in \partial \widehat{K}(v)\cap \partial S(K,v)$ for some $v\in V$.
 	\item $p\in C(K,v)$ for some $v\in V$.
 \end{enumerate}
 \begin{proof}
 	Suppose that there is a point $v\in V$ such that $p\in \partial S(K,v)$ and $p\not\in \partial \widehat{K}(v)$. Then $p\in\Int(\widehat{K}(v))$. By Corollary \ref{cap_equal_to_int_spkie}, we get $p\in \partial K\cap \Int(\widehat{K}(v))=C(K,v)$.
 \end{proof}
 
\end{lem}
 For any point $p\in\mathbb{E}^d$ and $\delta>0$, we denote by $B(p,\delta)$ the ball in $\mathbb{E}^d$ with center $p$ and radius $\delta$. 
 
 \begin{lem}\label{boundary_of_spike_ray_not_intersect_int_K}
 	For any $p\in\partial K\cap\partial\widehat{K}(v)\cap\partial(S(K,v))$, we have that the ray $vp$ does not intersect $\Int(K)$.
 \end{lem}
 \begin{proof}
  Let $p\in\partial K\cap\partial\widehat{K}(v)\cap\partial(S(K,v))$.	For the sake of contradiction, suppose that the ray $vp$ intersects $\Int(K)$. If the line segment $\overline{vp}$ does not intersect $\Int(K)$, then   $p\in C(K,v)\subset\Int(\widehat{K}(v))$ which is impossible, since $p\in \partial\widehat{K}(v)$. Now suppose that the line segment $\overline{vp}$ intersects $\Int(K)$. By convexity and continuity, there exists $\delta>0$ such that for each $q\in B(p,\delta)\cap K$, the line segment $\overline{vq}$ intersects $\Int(K)$. For $q\in\partial K$, let $H_q$ be a support hyperplane of $K$ at the point $q$. Denote by $H_q^+$ the closed half space that contains $K$, and $\partial H_q^+=H_q$. It is clear that $K$ and $v$ are contained in $\bigcap_{q\in B(p,\delta)\cap \partial K} H_q^+$. It follows that $S(K,v)$ is a subset of $\bigcap_{q\in B(p,\delta)\cap \partial K} H_q^+$. Furthermore, we have that $B(p,\delta)\setminus K$ does not overlap with $\bigcap_{q\in B(p,\delta)\cap \partial K} H_q^+$.  One can deduce that $B(p,\delta)$ does not contain any point in $S(K,v)$, which is a contradiction, since $p\in \partial(S(K,v))$.
 \end{proof}
 
 \begin{lem}\label{not_on_boundary_of_spike}
 	Let $K$ be a $d$-dimensional convex body. Let $V$ be a finite subset of $\mathbb{E}^d\setminus K$ that satisfies the condition in Lemma \ref{spiky_convex_condition}. Let $p\in \partial K\cap\partial\widehat{K}(V)$ and $u\in\mathbb{S}^{d-1}$. Suppose that $p\not\in \partial(S(K,v))$ for all $v\in V$. Then the direction $u$ illuminates $p$ with respect to $K$ if and only if $u$ illuminates $p$ with respect to $\widehat{K}(V)$.
 \end{lem}
 \begin{proof}
 	Since $p\not\in \partial(S(K,v))$ for all $v\in V$ and $V$ is finite, there exists $\delta>0$ such that
 	$B(p,\delta)\cap\widehat{K}(V)= B(p,\delta)\cap K$.  This implies our desired result.
 \end{proof}
 
  Let $V$ be a finite (or countable) subset of $\mathbb{E}^d\setminus K$ that satisfies the condition in Lemma \ref{spiky_convex_condition}. For any given $v\in V$ and $u\in\mathbb{S}^{d-1}$. Since $v\not\in \widehat{K}(V\setminus{v})$ and $\widehat{K}(V\setminus{v})$ is closed, there exists $\delta>0$ such that
  $B(v,\delta)\cap\widehat{K}(V)= B(v,\delta)\cap \widehat{K}(v)$. Hence,  the direction $u$ illuminates $v$ with respect to $\widehat{K}(V)$ if and only if $u$ illuminates $v$ with respect to $\widehat{K}(v)$.
  
\begin{lem}\label{sub_cap_body_must_be_illumnated}
	Let $K$ be a $d$-dimensional convex body. Let $V_1$ and $V_2$ be two finite subsets of $\mathbb{E}^d\setminus K$ such that both satisfy the condition in Lemma \ref{spiky_convex_condition}. Suppose that $K$ has a unique support hyperplane at each point in $$\bigcup_{v\in V_1}\left(\partial K\cap\partial\widehat{K}(v)\cap\partial(S(K,v))\right).$$   If $\widehat{K}(V_2)\subset \widehat{K}(V_1)$ and the cap body $\widehat{K}(V_1)$ is $m$-fold illuminated by a multiset $U$ of directions in $\mathbb{S}^{d-1}$, then $\widehat{K}(V_2)$ is also $m$-fold illuminated by $U$. Therefore,
	$$I^m\left(\widehat{K}(V_2)\right)\leq I^m\left(\widehat{K}(V_1)\right).$$
\end{lem}

\begin{proof}
	 Suppose that $p\in\partial\widehat{K}(V_2)$. Then either $p\in\partial K$, or $p\in S(K,v)$ for some $v\in V_2$.
	 
	 \textbf{Case 1:} $p\in S(K,v)$ for some $v\in V_2$. Since $\widehat{K}(V_2)\subset\widehat{K}(V_1)$, there exists $v' \in V_1$ such that $v\in S(K,v')$. Because $\widehat{K}(V_1)$ is $m$-fold illuminated by $U$, we have that $v'$ is $m$-fold illuminated by $U$ with respect to $\widehat{K}(V_1)$, and hence $v'$ is $m$-fold illuminated by $U$ with respect to $\widehat{K}(v')$. By Lemma \ref{illuminate_inside_spike}, one obtains that $v$ is $m$-fold illuminated by $U$ with respect to $\widehat{K}(v)$. It follows, by Lemma \ref{illuminate_cap}, that  $p$ is $m$-fold illuminated by $U$ with respect to $\widehat{K}(v)$. Since $\Int\left(\widehat{K}(v)\right)\subset\Int\left(\widehat{K}(V_2)\right)$,  $p$ is $m$-fold illuminated by $U$ with respect to $\widehat{K}(V_2)$.
	 
	\textbf{Case 2:} $p\in\partial K$. By Lemma \ref{possible_positions_of_point_on_boundary_of_K}, we consider the following three cases.
	 \begin{itemize}
	 	\item[] \textbf{Case 2.1:} $p\not\in \partial S(K,v)$ for all $v\in V_1$. Then $p\in\partial \widehat{K}(V_1)$, and hence $p$ is $m$-fold illuminated by $U$ with respect to $\widehat{K}(V_1)$. By Lemma \ref{not_on_boundary_of_spike}, we obtain that  $p$ is $m$-fold illuminated by $U$ with respect to $K$.
	 	
	 	\item[] \textbf{Case 2.2:} $p\in\partial \widehat{K}(v)\cap\partial S(K,v)$ for some $v\in V_1$. Then $K$ has a unique support hyperplane $H_p$ at the point $p$. By Lemma \ref{boundary_of_spike_ray_not_intersect_int_K}, we have that the ray $vp$ does not intersect $\Int(K)$. Since $v$ is $m$-fold illuminated by $U$ with respect to $\widehat{K}(v)$, it follows, from Remark \ref{replace_smooth_condition}, that $p$ is $m$-fold illuminated by $U$ with respect to $K$.
	 	
	 	\item[] \textbf{Case 2.3:} $p\in C(K,v)$ for some $v\in V_1$. Since $v$ is $m$-fold illuminated by $U$ with respect to $\widehat{K}(v)$, by Lemma \ref{illuminate_cap}, we obtain that $p$ is $m$-fold illuminated by $U$ with respect to $K$.
	 \end{itemize}
	 From the above three sub-cases, we can conclude that if $p\in\partial K$, then $p$ is $m$-fold illuminated by $U$ with respect to $\widehat{K}(V_2)$, since $\Int(K)\subset \Int\left(\widehat{K}(V_2)\right)$.
\end{proof}
\begin{rem}
	In Lemma \ref{sub_cap_body_must_be_illumnated}, $V_2$ could be empty set. 
\end{rem}

\section{Main results}

\subsection{A lower bound of $m$-fold illumination numbers of $d$-dimensional convex bodies}\label{subsection_lower_bound}
Let $K$ be a $d$-dimensional convex body. 
Suppose that $K$ is $m$-fold illuminated by a multiset $U$ of directions in $\mathbb{S}^{d-1}$. Let $V=\{v_1,\ldots,v_{d-1}\}\subset U$, and let $H$ be a hyperplane parallel to all directions in $V$. Denote by $H_1$ and $H_2$ two distinct support hyperplanes of $K$ parallel to $H$. Let $p_1\in H_1\cap K$ and $p_2\in H_2\cap K$. Then, the points $p_1$ and $p_2$ both are not illuminated by $\{v_1,\ldots,v_{d-1}\}$. Furthermore,  $p_1$ and $p_2$ cannot be simultaneously illuminated  by any direction $u\in U$. Note that $p_1$ and $p_2$ both must be illuminated by at least $m$ directions. It follows that $U$ contains at least $(d-1)+2m$ directions, and hence $$I^m(K)\geq 2m+(d-1).$$ This completes the proof of  Theorem \ref{main_lower_bound}.

\subsection{$m$-fold illumination numbers of two-dimensional convex bodies with smooth boundary}\label{subsection_smooth}
Suppose that $K$ is a $2$-dimensional convex body with smooth boundary. One can construct an equiangular  $(2m+1)$-sided polygon $P=p_1p_2\cdots p_{2m+1}$ such that for $i=1,2,\ldots,2m+1$, the side $p_{i}p_{i+1}$ of $P$ touches the boundary of $K$, here we define $p_{k+2m+1}=p_k$ for all integers $k$. Clearly, the polygon $P$ can be seen as the cap body $\widehat{K}(\{p_1,p_2,\ldots,p_{2m+1}\})$. Let $v_i$ be the intersection of the lines $p_{i}p_{i+1}$  and $p_{i+m}p_{i+m+1}$. Observe that each exterior angle of $P$ is equal to $2\pi/(2m+1)$, and  $2m\pi/(2m+1)<\pi$. One can deduce that  $p_{i+1}, p_{i+2},\ldots,p_{i+m}\in S(K,v_i)$. Let $u_i$ be a direction that illuminates $v_i$ with respect to $\widehat{K}(v_i)$.  By Lemma \ref{illuminate_inside_spike}, we have that the direction $u_i$ illuminates the points $p_{i+k}$ with respect to $\widehat{K}(p_{i+k})$, for all $k=1,2,\ldots,m$. Therefore, by Lemma \ref{illuminate_closed_cap}, one obtains that $u_i$ illuminates all points in the closed cap $\overline{C}(K,p_{i+k})$ with respect to $K$, for $k=1,2,\ldots,m$. Hence, $\overline{C}(K,p_i)$ is illuminated by the direction $u_{i-k}$ with respect to $K$, for all $k=1,2,\ldots,m$. Note that $$\partial K=\bigcup_{i=1}^{2m+1}\overline{C}(K,p_i).$$
It follows immediately that $K$  is $m$-fold illuminated by the set of directions $\{u_1,u_2,\ldots,u_{2m+1}\}$. Hence, $I^m(K)\leq 2m+1$. On the other hand, by Theorem \ref{main_lower_bound}, we know that $I^m(K)\geq 2m+1$. This completes the proof of Theorem \ref{main_smooth_two_dim}. 

\subsection{$m$-fold illumination numbers of convex polygons}\label{subsection_polygon}
Let $P$ be a convex polygon. If $P$ is a parallelogram, then $I^m(P)=4m$. If $P$ is a triangle, then $I^m(P)=3m$. From the proof of Theorem \ref{main_smooth_two_dim}, one may observe that $I^m(P)=2m+1$, if $P$ is an equiangular convex $(2m+1)$-sided polygon. In general,  by modifying the arguments used in the proof of Theorem \ref{main_smooth_two_dim} (see Section \ref{subsection_smooth}), one can easily prove the following useful lemma.
\begin{lem}\label{convex_2m_1_polygon}
	Let $m\geq 2$. Let $P=p_1p_2\cdots p_{2m+1}$ be a convex $(2m+1)$-sided polygon. For any integer $i$, we define $p_{i+2m+1}=p_i$, and denote by $\alpha_i$ the radian measure of exterior angle of $P$ at the vertex $p_i$. Suppose that $\alpha_k+\alpha_{k+1}+\cdots+\alpha_{k+m-1}<\pi$, for all integer $k$. Then $I^m(P)=2m+1$.
\end{lem}

In general, one can prove the following lemma.
\begin{lem}
	Let $m\geq 2$. Let $P=p_1p_2\cdots p_{n}$ be a convex $n$-sided polygon, where $n\geq 2m+1$. For any integer $i$, we define $p_{i+n}=p_i$, and denote by $\alpha_i$ the radian measure of exterior angle of $P$ at the vertex $p_i$. Suppose that there exists a  sequence of integers $\left(n_k\right)_{k\in\mathbb{N}}$ such that
	\begin{enumerate}[label=(\roman*)]
		\item $0=n_0<n_1<n_2<\cdots<n_{2m}$ and $n_{i+(2m+1)k}=n_i+nk$ for all $i,k=0,1,2,\ldots$,
		\item $\beta_k+\beta_{k+1}+\cdots+\beta_{k+m-1}<\pi$ for all integer $k$, where $$\beta_k=\sum_{i=n_{k-1}+1}^{n_k}\alpha_i.$$ 
	\end{enumerate} 
	Then $I^m(P)=2m+1$.
\end{lem}

\begin{lem}\label{illumination_number_of_K_contained_in_polygon}
	Let $K$ be a $2$-dimensional convex body. Let $P$ be a convex $n$-sided polygon containing $K$ such that each side of $P$ touches the boundary of $K$, and $I^m(P)\leq l$. Suppose that for each side $E$ of $P$, if $E\cap \partial K$ contains only one point $p$, then $K$ has a unique support line at the point $p$. Then $I^m(K)\leq l$.
\end{lem}

\begin{proof}
	Assume that $P=p_1p_2\cdots p_{n}$ and let $p_{n+1}=p_1$. Since each side of $P$ touches the boundary of $K$, we have that $P=\widehat{K}(\{p_1,p_2,\ldots,p_{n}\})$. Denote by $[s,t]$ the line segment with endpoints $s$ and $t$. For $i=1,2,\ldots,n$, let $[s_i,t_i]=[p_i,p_{i+1}]\cap\partial K$, where $s_i$ lies between $p_i$ and $t_i$. It is clear that
	$$\partial K=\bigcup_{i=1}^{n}C(K,p_i)\cup \bigcup_{i=1}^{n}[s_i,t_i].$$
	Since $I^m(P)\leq l$, there exists a multiset $U$ of directions in $\mathbb{S}^1$ such that $P$ is $m$-fold illuminated by $U$, and $|U|\leq l$. Then for $i=1,2,\ldots,n$, the vertex $p_i$ is $m$-fold illuminated by $U$ with respect to $P$. By Lemma \ref{illuminate_cap}, we have that each point in $C(K,p_i)$ is $m$-fold illuminated by $U$ with respect to $K$. 
	
	Now consider $p\in[s_i,t_i]$, for some $i=1,2,\ldots,n$.  We will show that $p$ is $m$-fold illuminated by $U$ with respect to $K$.
	\begin{itemize}
		\item[] \textbf{Case 1:} $s_i=t_i$. Then, from the given condition, $K$ has a unique support line at the point $p$. Note that $p\in\overline{C}(K,p_i)\setminus C(K,p_i)$ and $p_i$ is $m$-fold illuminated by $U$ with respect to $P$. From Remark \ref{replace_smooth_condition}, we obtain that $p$ is $m$-fold illuminated by $U$ with respect to $K$.
		\item[] \textbf{Case 2:} $s_i\neq t_i$. If $p\in(s_i,t_i)$, then $K$ has a unique support line at the point $p$. By the same argument used in Case 1, one obtains that $p$ is $m$-fold illuminated by $U$ with respect to $K$. Now suppose that $p=s_i$. For any direction $u$ that illuminates $p_i$ with respect to $P$. From the proof of Lemma \ref{illuminate_cap}, there exists $q\in \Int(K)$ such that $u=\lambda\cdot(q-p_i)$ for some $\lambda>0$.  Denote by $L$ the line passing through $s_i$ and parallel to $u$. It is clear that $t_i$ and $q$ lie on the opposite sides of the line $L$. Note that $t_i\in K$ and $q\in\Int(K)$. It follows that $L$ intersect $\Int(K)$. This implies that $u$ illuminates $s_i$ with respect to $K$. Since $p_i$ is $m$-fold illuminated by $U$ with respect to $P$, we have that $s_i$ is $m$-fold illuminated by $U$ with respect to $K$. If $p=t_i$, then similar to the case $p=s_i$, we may consider the illumination at the point $p_{i+1}$, and obtain that  $t_i$ is $m$-fold illuminated by $U$ with respect to $K$.
	\end{itemize}
	In conclusion, we have that each point in $\partial K$ is $m$-fold illuminated by $U$ with respect to $K$, and hence $I^m(K)\leq |U|\leq l$. 
\end{proof}

\begin{rem}
	In Lemma \ref{illumination_number_of_K_contained_in_polygon}, if $l=2m+1$, then we have that $I^m(K)\leq 2m+1$. On the other hand, by Theorem \ref{main_lower_bound}, we know that $I^m(K)\geq 2m+1$. Therefore, $I^m(K)=2m+1$.
\end{rem}
Now we consider the $m$-fold illumination numbers of regular convex polygons.

\begin{lem}\label{illumination_number_of_regular_polygon}
	Let $K$ be a regular convex $n$-sided polygon, where $n\geq 3$. Then
	$$I^m(K)=\left\lceil\frac{mn}{\left\lfloor\frac{n-1}{2}\right\rfloor}\right\rceil.$$
\end{lem}

\begin{proof}
	First we will show that $$I^m(K)\geq \frac{mn}{\left\lfloor\frac{n-1}{2}\right\rfloor}.$$
	Suppose that $K$ is $m$-fold illuminated by a multiset $U$ of directions in $\mathbb{S}^1$. Denote by $I(K,U)$ the multiset of all pairs of $(p,u)$ such that $p$ is a vertex of $K$, $u\in U$, and $u$ illuminates $p$ with respect to $K$. Since $K$ is $m$-fold illuminated by $U$, each vertex of $K$ must be illuminated by at least $m$ directions from $U$. Hence, we have
	$$|I(K,U)|\geq mn,$$
	where $|\cdot|$ refers to the sum of multiplicities of all elements in the corresponding set. On the other hand, since $K$ is a regular convex $n$-sided polygon, we have that for any $u\in U$, the direction $u$ can simultaneously illuminate at most $\left\lfloor\frac{n-1}{2}\right\rfloor$ vertices of $K$, and hence
	$$|I(K,U)|\leq |U|\cdot \left\lfloor\frac{n-1}{2}\right\rfloor.$$
	It follows that
	$$|U|\geq \frac{mn}{\left\lfloor\frac{n-1}{2}\right\rfloor}.$$
	
	Now we will show that there exists a multiset  $U$ of directions in $\mathbb{S}^1$ with $$|U|=\left\lceil\frac{mn}{\left\lfloor\frac{n-1}{2}\right\rfloor}\right\rceil$$ such that $K$ is $m$-fold illuminated by $U$. Suppose that $K=p_1p_2\cdots p_n$. For any integer $i$, we define $p_{i+n}=p_i$, and let $u_i$ be a direction that illuminates $$p_i,p_{i+1},\ldots,p_{i+\left\lfloor\frac{n-1}{2}\right\rfloor-1}$$
    with respect to $K$. Choose
    $$U=\left\{u_i\mid i=k\cdot\left\lfloor\frac{n-1}{2}\right\rfloor ,~k=0,1,\ldots,\left\lceil\frac{mn}{\left\lfloor\frac{n-1}{2}\right\rfloor}\right\rceil-1\right\}.$$
    Since 
    $$\left\lfloor\frac{n-1}{2}\right\rfloor\cdot\left\lceil\frac{mn}{\left\lfloor\frac{n-1}{2}\right\rfloor}\right\rceil\geq mn,$$
    one obtains that each vertex $p_i$ of $K$ is $m$-fold illuminated by $U$ with respect to $K$. This implies that the polygon $K$ is $m$-fold illuminated by $U$.
\end{proof}

If $n$ is odd and $n\geq 2m+1$, then
$$\frac{mn}{\left\lfloor\frac{n-1}{2}\right\rfloor}=\frac{2mn}{n-1},$$
and
$$2m<\frac{2mn}{n-1}\leq 2m+1.$$
If $n$ is even and $n\geq 4m+2$, then
$$\frac{mn}{\left\lfloor\frac{n-1}{2}\right\rfloor}=\frac{2mn}{n-2},$$
and
$$2m<\frac{2mn}{n-2}\leq 2m+1.$$
Using the result in Lemma \ref{illumination_number_of_regular_polygon}, we complete the proof of Theorem \ref{main_regular_polygon}.

\subsection{An upper bound of $m$-fold illumination numbers of $\mathbb{B}^d$} \label{subsection_upper_bound_ball}
Now we consider the $m$-fold illumination numbers of $d$-dimensional unit balls $\mathbb{B}^d$. Denote by $\langle\cdot,\cdot\rangle_d$ the dot product of Euclidean space $\mathbb{E}^d$. Let $u\in\mathbb{S}^{d-1}$ and $p\in \partial \mathbb{B}^d=\mathbb{S}^{d-1}$. It is easy to see that the direction $u$ illuminates the point $p$ with respect to $\mathbb{B}^d$ if and only if $\langle u,p\rangle_d<0$.
\begin{lem}\label{three_dim_ball_upper_bound}
	For any positive integers $m$, we have
	$$ I^m(\mathbb{B}^3)\leq 2m+1+\left\lceil\frac{m}{2}\right\rceil.$$
\end{lem}
\begin{proof}
Suppose that $0<\varepsilon<\cos\frac{m\pi}{2m+1}\leq\frac12$.
Let $$w_1=w_2=\cdots=w_{\lceil\frac{m}{2}\rceil}=(0,0,-1).$$
 For any integer $i$, let $$u_i=\left(-\sqrt{1-\varepsilon_i^2}\cos\frac{2i\pi}{2m+1},-\sqrt{1-\varepsilon_i^2}\sin\frac{2i\pi}{2m+1}, \varepsilon_i\right),$$
 where for $j=0,1,2,\ldots,2m$, we define
 $$\varepsilon_j=
 \begin{cases}
 \varepsilon, & j~\mbox{is odd}\\
 -\varepsilon^2, & j~\mbox{is even},
 \end{cases} 
$$
 and $\varepsilon_{j+2m+1}=\varepsilon_j$, for all integer $j$. Then $u_{j}=u_i$ if and only if $j-i$ is divided by $2m+1$.

For any $p=(\sqrt{1-z^2}\cos\theta,\sqrt{1-z^2}\sin\theta,z)\in\partial\mathbb{B}^3$, where $\theta\in[0,2\pi)$ and $z\in[-1,1]$. We consider the following three cases. 

\textbf{Case} $z\in[-1,-\sqrt{1-\varepsilon^2})$.  Then
$$\frac{\sqrt{1-z^2}}{-z}<\frac{\varepsilon}{\sqrt{1-\varepsilon^2}}.$$
For $i=1,3\ldots,2m-1$,
\begin{align*}
	\langle u_i,p\rangle &=-\sqrt{1-\varepsilon^2}\sqrt{1-z^2}\cos\left(\frac{2i\pi}{2m+1}-\theta\right)+\varepsilon z\\
	&\leq \sqrt{1-\varepsilon^2}\sqrt{1-z^2}+\varepsilon z <0.
\end{align*}

 \textbf{Case} $z\in\left[-\sqrt{1-\varepsilon^2},\frac12\right]$. Suppose that $i$ is an integer such that $$\frac{-m+(2m+1)\pi^{-1}\theta}{2}\leq i\leq \frac{m+(2m+1)\pi^{-1}\theta}{2}.$$
 Then
 $$-\frac{m\pi}{2m+1}\leq\frac{2i\pi}{2m+1}-\theta\leq \frac{m\pi}{2m+1}.$$
 Note that there are at least $m$ numbers of integers $i$ satisfy the above condition.
  Consider
 \begin{align*}
 	\langle u_i,p\rangle &=-\sqrt{1-\varepsilon_i^2}\sqrt{1-z^2}\cos\left(\frac{2i\pi}{2m+1}-\theta\right)+\varepsilon_i z\\
 	&\leq -\sqrt{1-\varepsilon_i^2}\sqrt{1-z^2}\cos\left(\frac{m\pi}{2m+1}\right)+\varepsilon_i z\\
 	&< -\varepsilon\sqrt{1-\varepsilon_i^2}\sqrt{1-z^2}+\varepsilon_i z\\
 \end{align*}
 If $\varepsilon_i=\varepsilon$ and $z\in\left[-\sqrt{1-\varepsilon^2},0\right]$, then we have $-\varepsilon\sqrt{1-\varepsilon_i^2}\sqrt{1-z^2}+\varepsilon_i z<0$.
 If $\varepsilon_i=\varepsilon$ and $z\in\left(0,\frac12\right]$, then we have
 $$
 	\frac{z}{\sqrt{1-z^2}} \leq \frac{1}{\sqrt{3}}<\frac{\sqrt{3}}{2}
 			< \sqrt{1-\varepsilon^2},
 $$
 and hence $-\varepsilon\sqrt{1-\varepsilon^2}\sqrt{1-z^2}+\varepsilon z<0$. If $\varepsilon_i=-\varepsilon^2$ and $z\in\left[0,\frac12\right]$, then we easily get $-\varepsilon\sqrt{1-\varepsilon_i^2}\sqrt{1-z^2}+\varepsilon_i z<0$.
 If $\varepsilon_i=-\varepsilon^2$ and $z\in\left[-\sqrt{1-\varepsilon^2},0\right)$, then
 $$
 \frac{z}{\sqrt{1-z^2}} \geq \frac{-\sqrt{1-\varepsilon^2}}{\varepsilon}>\frac{-\sqrt{1-\varepsilon^4}}{\varepsilon},
 $$
 which implies that $-\varepsilon\sqrt{1-\varepsilon^4}\sqrt{1-z^2}-\varepsilon^2 z<0$.

\textbf{Case} $z\in\left(\frac12,1\right]$. Suppose that $i$ is an integer such that
$$\frac{-m+(2m+1)\pi^{-1}\theta}{2}\leq i\leq \frac{m+(2m+1)\pi^{-1}\theta}{2}.$$
and the (least non-negative) remainder of the division of $i$ by $2m+1$ is even. Note that there are at least $\left\lfloor\frac{m}{2}\right\rfloor$ numbers of integers $i$ satisfy the above conditions. Observe that
 \begin{align*}
	\langle u_i,p\rangle &=-\sqrt{1-\varepsilon^4}\sqrt{1-z^2}\cos\left(\frac{2i\pi}{2m+1}-\theta\right)-\varepsilon^2 z\\
	&\leq -\sqrt{1-\varepsilon^4}\sqrt{1-z^2}\cos\left(\frac{m\pi}{2m+1}\right)-\varepsilon^2 z\\
	&< -\varepsilon\sqrt{1-\varepsilon^4}\sqrt{1-z^2}-\varepsilon^2z<0.\\
\end{align*}
Furthermore, for $j=1,2,\ldots,\left\lceil\frac{m}{2}\right\rceil$, we have $\langle w_j,p\rangle =-z<0$.
	
From the above three cases, one can conclude that $\mathbb{B}^3$ is $m$-fold illuminated by the multiset of directions  $$U=\left\{u_0,u_1,\ldots,u_{2m},w_1,\ldots,w_{\lceil\frac{m}{2}\rceil}\right\}.$$
This implies that 
	$$ I^m(\mathbb{B}^3)\leq 2m+1+\left\lceil\frac{m}{2}\right\rceil.$$
\end{proof}

\textbf{Proof of Theorem \ref{main_ball_upper_bound}:}	We will use an induction on $d$. By Lemma \ref{three_dim_ball_upper_bound}, the result is true for $d=3$. Now we suppose that the result is true for $\mathbb{B}^d$. Recall that
$$I^m(\mathbb{B}^d)=C^m(\mathbb{B}^d),$$
where $C^m(\mathbb{B}^d)$ is the $m$-fold covering number of $\mathbb{B}^d$. Then
$$C^m(\mathbb{B}^d)\leq (d-1)m+1+\left\lceil\frac{m}{2}\right\rceil.$$
Suppose that $v_1,\ldots,v_n\in\mathbb{E}^d$ such that $\mathbb{B}^d$ is $m$-fold covered by $$\Int(\mathbb{B}^d)+v_1,\ldots,\Int(\mathbb{B}^d)+v_n.$$
and $n\leq (d-1)m+1+\left\lceil\frac{m}{2}\right\rceil$. Assume that
$$v_i=(v_{i1},\ldots,v_{id}),$$
for $i=1,2,\ldots,n$.
 Let
$$B=\{(x_1,\ldots,x_d,0)\in\mathbb{E}^{d+1}\mid x_1^2+\cdots+x_d^2\leq1\},$$
and
$$u_i=(v_{i1},\ldots,v_{id},0),$$
for $i=1,2,\ldots,n$. Then $B$ is $m$-fold covered by $$\relint(B)+u_1,\ldots,\relint(B)+u_n,$$
where $\relint(\cdot)$ refers to the relative interior of the corresponding set.
Let $H$ be the hyperplane in $\mathbb{E}^{d+1}$ passing through the origin and orthogonal to $(0,\ldots,0,1)$. Denote by $\rho$ the stereographic projection from $\mathbb{S}^d$ to $H$ with center of projection $(0,\ldots,0,1)$. Let 
$$\mathbb{S}^d_{-}=\{(x_1,\ldots,x_{d+1})\mid (x_1,\ldots,x_{d+1})\in\mathbb{S}^d~\mbox{and}~x_{d+1}\leq 0\}.$$
Clearly, $\rho^{-1}(B)=\mathbb{S}^d_{-}$. For $i=1,\ldots,n$, we denote by $c_i$ the center of the spherical  circle $\rho^{-1}(\relint(B)+u_i)$ on $\mathbb{S}^d$. Note that for each $i=1,\ldots,n$, $\rho^{-1}(\relint(B)+u_i)$ must be contained in the open hemisphere on $\mathbb{S}^d$ with center at $c_i$. Moreover, $\mathbb{S}^d_{-}$ is $m$-fold covered by $\rho^{-1}(\relint(B)+u_1),\ldots,\rho^{-1}(\relint(B)+u_n)$. It follows that $\mathbb{S}^d_{-}$ is $m$-fold illuminated by the multiset of directions $\{-c_1,\ldots,-c_n\}$,
with respect to $\mathbb{B}^{d+1}$. Let $w_1=\cdots=w_m=(0,\ldots,0,-1)$. Then $\mathbb{B}^{d+1}$ is $m$-fold illuminated by
$\{-c_1,\ldots,-c_n,w_1,\ldots,w_m\}$.
It follows that $$I^m(\mathbb{B}^{d+1})\leq n+m\leq dm+1+\left\lceil\frac{m}{2}\right\rceil.$$ This completes the proof of Theorem \ref{main_ball_upper_bound}.

\begin{rem}
	In fact, we have proved that $$I^m(\mathbb{B}^{d+1})\leq m+I^m(\mathbb{B}^d).$$
\end{rem}

\subsection{$m$-fold illumination numbers of some cap bodies of $\mathbb{B}^d$ in small dimensions}\label{subsection_cap_body_of_ball}
Recently, Ivanov\cite{Ivanov} and Bezdek\cite{Bezdek} studied the ($1$-fold) illumination numbers of cap bodies and spiky balls, respectively. In this subsection, we will study the $m$-fold illumination numbers of some cap bodies of $\mathbb{B}^d$, when $d=2,3$.

 In $\mathbb{E}^2$, it is known that(\cite{Boltyanskii},\cite{Gohberg},\cite{Hadwiger} and \cite{Levi})
$$
I(K)=
\begin{cases}
	4, & K~\mbox{is a parallelogram},\\
	3, & ~\mbox{others}.
\end{cases}
$$
Now we suppose that $K$ is a cap body of $\mathbb{B}^2$. If $K$ is a parallelogram, then  $I^m(K)=4m$. If $K$ is not a parallelogram, then 
$$I^m(K)\leq mI(K)=3m.$$
This upper bound is sharp, because $I^m(K)=3m$, provided $K$ is a triangle. On the other hand, from Lemma \ref{sub_cap_body_must_be_illumnated} and Theorem \ref{main_smooth_two_dim}, one gets $$I^m(K)\geq I^m(\mathbb{B}^2)=2m+1.$$
By Lemma \ref{sub_cap_body_must_be_illumnated} and Lemma \ref{convex_2m_1_polygon}, we have
\begin{lem}\label{contained_in_polygon}
	Let $P$ be a convex $(2m+1)$-sided polygon circumscribed about $\mathbb{B}^2$, i.e., each side of $P$ touches the boundary of  $\mathbb{B}^2$. Suppose that $P$ satisfies the condition in Lemma \ref{convex_2m_1_polygon}. If $K$ is a cap body of $\mathbb{B}^2$ such that $K\subset P$, then
	$$I^m(K)=2m+1.$$
\end{lem}

\begin{lem}
	 For any $v\in\mathbb{E}^2\setminus\mathbb{B}^2$. We have
	$$I^m\left(\widehat{\mathbb{B}^2}(v)\right)=2m+1.$$
\end{lem}
\begin{proof}
	If $m=1$, the result is obvious. Now we assume that $m\geq 2$. Denote by $\alpha$ the radian measure of exterior angle of $\widehat{\mathbb{B}^2}(v)$ at the point $v$. Choose a positive real number $\varepsilon$ such that $\varepsilon<\frac{\pi-\alpha}{m}$. Let $p_1=v$. Then one can construct a convex $(2m+1)$-sided polygon $P=p_1p_2\cdots p_{2m+1}$ circumscribed about $\mathbb{B}^2$ with exterior angle $\alpha_i$ at the vertex $p_i$ for $i=1,2,\ldots,2m+1$, where
	$$\alpha_1=\alpha, ~\alpha_2=\alpha_3=\cdots=\alpha_m=\frac{\pi-\alpha-\varepsilon}{m-1}$$
	and
	$$\alpha_{m+1}=\alpha_{m+2}=\frac\alpha2+\varepsilon, ~\alpha_{m+3}=\cdots=\alpha_{2m+1}=\frac{\pi-\alpha-\varepsilon}{m-1}.$$
	Observe that
	$$\alpha+(m-1)\cdot \frac{\pi-\alpha-\varepsilon}{m-1}=\pi-\varepsilon<\pi,$$
	$$(m-1)\cdot \frac{\pi-\alpha-\varepsilon}{m-1}+\frac\alpha2+\varepsilon=\pi-\frac\alpha2<\pi,$$
	and
	$$(m-2)\cdot \frac{\pi-\alpha-\varepsilon}{m-1}+2\left(\frac\alpha2+\varepsilon\right)=\frac{(m-2)\pi+\alpha+m\cdot\varepsilon}{m-1}<\pi.$$
	Therefore, $P$ satisfies the condition in Lemma \ref{convex_2m_1_polygon}. By our construction, it is obvious that $\widehat{\mathbb{B}^2}(v)\subset P$. From Lemma \ref{contained_in_polygon}, one obtains the desired result.
\end{proof}
\begin{rem}
	When the number of elements in $V$ is more than $1$, the calculation of $I^m\left(\widehat{\mathbb{B}^2}(V)\right)$ becomes more complicated and might be related to the results on convex polygon cases.
\end{rem}

Now we consider the $m$-fold illumination number of a cap body $K=\widehat{\mathbb{B}^3}(V)$, where $V$ is a finite subset of $\mathbb{E}^3\setminus\mathbb{B}^3$ such that $V$ satisfies the condition in Lemma \ref{spiky_convex_condition}. From Lemma \ref{sub_cap_body_must_be_illumnated} and Theorem \ref{main_lower_bound}, we have that
$$I^m(K)\geq I^m(\mathbb{B}^3)\geq 2m+2.$$
 One may observe that for any $v\in V$, the closed cap $\overline{C}(\mathbb{B}^3,v)$ is a closed spherical cap on the boundary of $\mathbb{B}^3$. It is not hard to show that a direction $u$ illuminates $v\in V$ with respect to $\widehat{\mathbb{B}^3}(V)$ if and only if $u$ illuminates $\overline{C}(\mathbb{B}^3,v)$ with respect to $\mathbb{B}^3$. Furthermore, for any $v,v'\in V$ where $v\neq v'$, if the sum of the spherical radii of $\overline{C}(\mathbb{B}^3,v)$ and $\overline{C}(\mathbb{B}^3,v')$ is not less than $\frac{\pi}{2}$, then $\overline{C}(\mathbb{B}^3,v)$ and $\overline{C}(\mathbb{B}^3,v')$ cannot be simultaneously illuminated by any direction with respect to $\mathbb{B}^3$. It follows, from Lemma \ref{illuminate_closed_cap}, that $v$ and $v'$ cannot be simultaneously illuminated by any direction with respect to $\widehat{\mathbb{B}^3}(V)$.  Ivanov\cite{Ivanov} proved that the  illumination number of a centrally symmetric cap body of $\mathbb{B}^3$ is at most 6.
This implies that $$I^m(K)\leq 6m,$$ when $K$ is a centrally symmetric cap body of $\mathbb{B}^3$. Furthermore, if $(e_1,e_2,e_3)$ is an orthonormal basis for $\mathbb{E}^3$, then for $i=1,2,3$, the spherical radius of $\overline{C}(\mathbb{B}^3,\pm\sqrt{2}e_i)$ is equal to $\frac{\pi}{4}$. It follows that
$$I^m\left(\widehat{\mathbb{B}^3}(\{\pm\sqrt{2}e_1,\pm\sqrt{2}e_2,\pm\sqrt{2}e_3\})\right)=6m.$$ 
In general, we have two lemmas below.
\begin{lem}\label{illumination_number_of_cap_body_regular_top_only}
Let $V=\{q_1,q_2,...,q_{n},q_{n+1}\}$, where
$$q_i=\left(\sec\frac{\pi}{n}\cos\frac{2i\pi}{n},\sec\frac{\pi}{n}\sin\frac{2i\pi}{n},0\right),$$
for $i=1,2,\ldots,n$ and
$$q_{n+1}=\left(0,0,\sec\frac{(n-2)\pi}{2n}\right).$$
Then
$$I^m\left(\widehat{\mathbb{B}^3}(V)\right)=m+\left\lceil\frac{mn}{\left\lfloor\frac{n-1}{2}\right\rfloor}\right\rceil.$$
\end{lem}
\begin{proof}
	 For $i=1,2,\ldots,n$, if a direction $u=(x,y,z)\in\mathbb{E}^3\setminus\{(0,0,0)\}$ illuminates $q_i$ with respect to $\widehat{\mathbb{B}^3}(V)$, then by symmetry, the direction $(x,y,-z)$ also illuminates $q_i$ with respect to $\widehat{\mathbb{B}^3}(V)$. Hence, by convexity, the direction $(x,y,0)$ illuminates $q_i$ with respect to $\widehat{\mathbb{B}^3}(V)$. Clearly, we have $x^2+y^2\neq 0$. 
	 
	 Now suppose that $W$ is a multiset of directions in $\mathbb{E}^3\setminus\{(0,0,0)\}$ such that for each $i=1,2,\ldots,n$, the vertex $q_i$ is $m$-fold illuminated by $W$ with respect to $\widehat{\mathbb{B}^3}(V)$. Then, by the above discussion, there exists a multiset $W'$ of directions parallel to the $XY$ plane such that $|W'|=|W|$, and for each $i=1,2,\ldots,n$, the vertex $q_i$ is $m$-fold illuminated by $W'$ with respect to the polygon $q_1q_2\cdots q_n$. 
	 Since $q_1q_2\cdots q_n$ forms a regular convex polygon, by Lemma \ref{illumination_number_of_regular_polygon}, we have that  $$|W|=|W'|\geq \left\lceil\frac{mn}{\left\lfloor\frac{n-1}{2}\right\rfloor}\right\rceil$$ Moreover, for each $i=1,2,\ldots,n$, the sum of spherical radii of $\overline{C}(\mathbb{B}^3,q_{n+1})$ and $\overline{C}(\mathbb{B}^3,q_i)$ is equal to $\frac\pi2$. Hence, $q_{n+1}$ and $q_i$ cannot be simultaneously illuminated by any direction  with respect to $\widehat{\mathbb{B}^3}(V)$. It follows that
	$$I^m\left(\widehat{\mathbb{B}^3}(V)\right)\geq m+\left\lceil\frac{mn}{\left\lfloor\frac{n-1}{2}\right\rfloor}\right\rceil.$$
	Let $U$ be a multiset of directions parallel to the $XY$ plane such that 
	$$|U|=\left\lceil\frac{mn}{\left\lfloor\frac{n-1}{2}\right\rfloor}\right\rceil,$$
	and $q_i$ is $m$-fold illuminated by $U$ with respect to $\widehat{\mathbb{B}^3}(V)$, for $i=1,2\ldots,n$. Then there exists a sufficiently small positive number $\varepsilon$, such that $\widehat{\mathbb{B}^3}(V)$ is $m$-fold illuminated by the multiset of directions
	$$\{u+(0,0,\varepsilon)\mid u\in U\}\cup\{m\cdot (0,0,-1)\},$$
	where $m$ is the multiplicity of $(0,0,-1)$.
	
\end{proof}

Similar to Lemma \ref{illumination_number_of_cap_body_regular_top_only}, one can easily obtain the following result.
\begin{lem}\label{illumination_number_of_cap_body_regular_top_bottom}
	Let $V=\{q_1,q_2,...,q_{n},q_{n+1}q_{n+2}\}$, where
	$$q_i=\left(\sec\frac{\pi}{n}\cos\frac{2i\pi}{n},\sec\frac{\pi}{n}\sin\frac{2i\pi}{n},0\right),$$
	for $i=1,2,\ldots,n$ and
	$$q_{n+1}=\left(0,0,\sec\frac{(n-2)\pi}{2n}\right), ~q_{n+2}=\left(0,0,-\sec\frac{(n-2)\pi}{2n}\right).$$
	Then
	$$I^m\left(\widehat{\mathbb{B}^3}(V)\right)=2m+\left\lceil\frac{mn}{\left\lfloor\frac{n-1}{2}\right\rfloor}\right\rceil.$$
\end{lem}

\section{Conclusion}
From the above discussions and results, the following problems naturally arise.
\begin{prob}
	 Find all possible values of the $m$-fold illumination numbers of $2$-dimensional convex bodies, for any given $m$.
\end{prob}
\begin{rem}
	We know that
	$$I^m(K)=\left\lceil\frac{mn}{\left\lfloor\frac{n-1}{2}\right\rfloor}\right\rceil,$$
	for a regular convex $n$-sided polygon $K$.
\end{rem}
\begin{prob}
	Find an algorithm to determine $I^m(P)$, for any given convex polygon $P$.
\end{prob}

\begin{prob}
	Determine $I^m(\mathbb{B}^d)$, for $d\geq 3$ and $m\geq 2$.
\end{prob}
\begin{rem}
	Note that $I^m(\mathbb{B}^2)=2m+1$ and $I(\mathbb{B}^d)=d+1$.	Furthermore, by Theorem \ref{main_lower_bound} and Theorem \ref{main_ball_upper_bound}, we know that $I^2(\mathbb{B}^3)=6$.
\end{rem}

\begin{prob}
	For any $3$-dimensional convex body $K$ with smooth boundary. Is it true that $I^2(K)=6$? Is it true that $I^m(K)=I^m(\mathbb{B}^3)$?
\end{prob}
\begin{prob}
	Improve some known upper bounds of the $m$-fold illumination numbers of $d$-dimensional convex bodies $K$. 
\end{prob}

\begin{rem}
	From the results in \cite{Papadoperakis} and \cite{Prymak}, we have that for any convex body $K$ in $\mathbb{E}^3, \mathbb{E}^4, \mathbb{E}^5$, and $\mathbb{E}^6$, we have
	$$I^m(K)\leq 16m,~I^m(K)\leq 96m,~I^m(K)\leq 1091m,~I^m(K)\leq 15373m,$$
	respectively. Furthermore, Huang's result\cite{Huang} implies that for any $d$-dimensional convex body $K$,
	$$I^m(K)\leq mc_14^de^{-c_2\sqrt{d}},$$
	for some universal constants $c_1$ and $c_2$. 
\end{rem}

\end{document}